\documentclass[11pt]{article}
\usepackage{amsthm}
\usepackage{amsmath}
\usepackage{amscd}

\usepackage[all]{xy}
\usepackage{tikz}
\usetikzlibrary{positioning} 
\usetikzlibrary{arrows,decorations.pathmorphing,shapes} 
\usepackage{tikz-cd}

\usepackage{graphicx, color}
\usepackage[T1]{fontenc}
\usepackage{caption}
\usepackage[latin2]{inputenc}
\usepackage[mathscr]{eucal}
\usepackage{indentfirst}
\usepackage{graphicx}
\usepackage{graphics}
\usepackage{pict2e}
\usepackage{epic}
\numberwithin{equation}{section}
\usepackage[margin=2.9cm]{geometry}
\usepackage{epstopdf} 
\usepackage{latexsym}
\usepackage[hyperfootnotes=false]{hyperref}
\usepackage{xargs}  

\usepackage{titlefoot}

\usepackage{float}
\usepackage{amsmath, amsfonts, amssymb, amsxtra, amsthm, mathrsfs}
\usepackage{enumitem}

\usepackage[colorinlistoftodos,prependcaption,textsize=tiny]{todonotes}

\newcommandx{\unsure}[2][1=]{\todo[linecolor=red,backgroundcolor=red!25,bordercolor=red,#1]{#2}}
\newcommandx{\change}[2][1=]{\todo[linecolor=blue,backgroundcolor=blue!25,bordercolor=blue,#1]{#2}}
\newcommandx{\info}[2][1=]{\todo[linecolor=OliveGreen,backgroundcolor=OliveGreen!25,bordercolor=OliveGreen,#1]{#2}}
\newcommandx{\improvement}[2][1=]{\todo[linecolor=Plum,backgroundcolor=Plum!25,bordercolor=Plum,#1]{#2}}

\usepackage[calcwidth]{titlesec}
\newcommand{\periodafter}[1]{#1.}
\titleformat{\paragraph}[runin]{\bfseries}{\theparagraph}{}{\periodafter}

\theoremstyle{plain}
\newtheorem{Thm}{Theorem}[section]
\newtheorem{Lemma}[Thm]{Lemma}
\newtheorem{Cor}[Thm]{Corollary}
\newtheorem{Prop}[Thm]{Proposition}

 \theoremstyle{definition}

\newtheorem{?}[Thm]{Problem}

\newtheorem{Ex}{Example}
\newtheorem*{Rmk*}{Remark}
\newenvironment{Ex*}
 {\pushQED{\qed}\Ex}
 {\popQED\endEx}

\allowdisplaybreaks

\usepackage{tocloft}
\usepackage{authblk}

\cftpagenumbersoff{part}

\usepackage{graphicx} 

\title{\bfseries Intersection Points of Closed Geodesics on Hyperbolic Surfaces of Finite Area}
\author{Tina Torkaman}
\date{\today}

\begin{document}
\date{}
\maketitle

\small{Published in \emph{International Mathematics Research Notices (IMRN)}.}
\bigskip

\begin{abstract}
Let $X$ be a complete hyperbolic surface of finite area. We establish that the intersection points of closed geodesics with length $\leq T$ are equidistributed on $X$ as $T \to \infty$.
\end{abstract}

\tableofcontents


\section{Introduction}
Consider a complete hyperbolic surface $X$ of finite area. Let $\mathcal{G}$ denote the set of all closed geodesics in $X$ and $\mathcal{G}_T$ the set of all closed geodesics of length $\leq T$. This paper establishes that the intersection points, including the self-intersections, of the closed geodesics in $\mathcal{G}_T$ are equidistributed with respect to the area measure on $X$ as $T \to \infty$. It is well known that closed geodesics are equidistributed in $T_1(X)$, the unit tangent bundle of $X$; see Theorem \ref{thm: Margulis.equi} or \cite[Thm. 9.1]{margulis.aspect}\cite{Bwn.equi}\cite{roblin}. We use the theory of geodesic currents to extend this to a similar result involving the intersection points of closed geodesics.

 For $\gamma_1,\gamma_2 \in \mathcal{G}$, define the intersection measure $I(\gamma_1,\gamma_2)$ on $X$ to be: 
 \begin{equation}
 I(\gamma_1,\gamma_2):=\sum \limits_{p\in \gamma_1 \cap \gamma_2} m_p\delta_p,
 \end{equation}
  where the sum is over all the transverse intersection points of $\gamma_1$ and $\gamma_2$, $\delta_p$ is the delta measure at the point $p$, and $m_p$ is the multiplicity at $p$. Specifically, $m_p=1$ unless the closed geodesics pass through the point $p$ multiple times (see \S \ref{sec: Pre}).
  The intersection number of $\gamma_1$ and $\gamma_2$, $i(\gamma_1,\gamma_2)$, is defined as the total volume $I(\gamma_1,\gamma_2)(X)$.
  Similarly, we can define intersection measure and intersection number of multi-geodesics $\sum \limits_i a_i\gamma_i$, where $a_i>0$ and $\gamma_i \in \mathcal{G}$. 
   
 Let $|\mu|$ denote the total volume of a measure $\mu$. Let $\alpha$ be the area measure on $X$ obtained from the hyperbolic metric. Then $\alpha(X)=area(X)=2\pi(2g-2+n)$ where $g$ is the genus of $X$ and $n$ the number of cusps. Define the multi-geodesic $\gamma_T$ to be: 
 \begin{equation}
 \gamma_T := \sum \limits_{\gamma \in \mathcal{G}_T} \gamma.
 \end{equation}

 The notation $C_c^*(X)$ denotes the dual space of the set of all compactly supported continuous functions on $X$.

 The main result is as follows:
 
 \begin{Thm}\label{theorem: main}
 Let $X$ be a complete hyperbolic surface of finite area. The intersection points of closed geodesics are equidistributed on $X$. In other words, we have
 
 \begin{equation}
 \frac{I(\gamma_T,\gamma_T)}{|I(\gamma_T,\gamma_T)|} \to \frac{\alpha}{area(X)},
 \end{equation}
 in $C_c^*(X)$ as $T \to \infty$.
 \end{Thm}
 
  The proof of Theorem \ref{theorem: main} also gives the growth rate of the intersection numbers between all the closed geodesics in $\mathcal{G}_T$, as $T \to \infty$. Let $\ell(\gamma)$ denote the length of $\gamma$.
  
  \begin{Cor}\label{Cor: main} As $T \to \infty$, we have
  \begin{equation}
  i(\gamma_T,\gamma_T) \sim \frac{1}{\pi^2(2g-2+n)}(\sum \limits_{\gamma \in \mathcal{G}_T} \ell(\gamma))^2 \sim \frac{e^{2T}}{\pi^2(2g-2+n)}.
  \end{equation}
  \end{Cor}
  Here, $\sim$ indicates that their ratio tends to $1$ as $T \to \infty$.

\paragraph{\textbf{Idea of the proof}}
The proof has two main parts:

\begin{itemize}
    \item \emph{Part $1$:} We first obtain a result weaker than Theorem \ref{theorem: main}. More precisely, we show:
    
    \begin{Thm} \label{thm: weaker.limit}
    As $T \to \infty$, we have
    \begin{equation}
        \frac{I(\gamma_T,\gamma_T)}{\ell(\gamma_T)^2} \to \frac{\alpha}{2\pi^3(2g-2+n)^2},
    \end{equation}
    in $C_c^*(X)$.
    \end{Thm}
        
    Theorem \ref{thm: weaker.limit} implies Theorem \ref{theorem: main} when $X$ is compact. However, the measures in the converging sequence in Theorem \ref{thm: weaker.limit} are not necessarily probability measures, and when $X$ has a cusp, there might be escape of mass to infinity (into the cusps). Therefore, Theorem \ref{thm: weaker.limit} is not sufficient to imply Theorem \ref{theorem: main}. The second part of the proof addresses this issue. 
    
    We apply the theory of geodesic currents to prove Theorem \ref{thm: weaker.limit}. For more details on geodesic currents, see $\S \ref{sec: Pre}$ and \cite{Bon.gc.3}\cite{Bon.gc}. For the proof of Theorem \ref{thm: weaker.limit}, see \S \ref{sec: equid.}.
    
    \item \emph{Part $2$:} We establish the \emph{tightness} of the measures within the converging sequence described in Theorem \ref{thm: weaker.limit}, confirming that there is no escape of mass into the cusp(s). This result implies Theorem \ref{theorem: main}.

     To prove the tightness, we control the contribution to the intersection number coming from intersections in the cusp(s). It is important to note that intersection number behaves differently in the presence of cusps. That is, when $X$ is compact, we have $i(\alpha,\beta)\leq c(X)\ell(\alpha)\ell(\beta)$ for a constant $c(X)>0$. However, when $X$ has a cusp(s), the intersection number can grow exponentially with respect to length. See \cite{Basmj} \cite{Tina.sys}. 
    Thus, a small measure (of a geodesic current) inside a cuspidal neighborhood does not necessarily imply a small intersection number. See Example \ref{ex: not.cont.int}. 
    
    We prove an upper bound to the number of self-intersection points of $\gamma_T$ within a cuspidal neighborhood $Y$ by considering the geodesic arcs in $Y$ separately, based on their winding number around the cusp (see \S $4$ for the definition of winding number). More precisely, let $Y$ be the horoball neighborhood of a cusp with horocycle boundary $\partial Y$ of length $1$. An $n-$excursion (or simply $n-exc$) corresponding to $Y$ is defined as an arc within $Y$ whose endpoints lie on $\partial Y$ and whose winding number around the cusp is $n$. Let $E_n(\gamma)$ denote the number of all $n-exc$'s of $\gamma \in \mathcal{G}$. Similarly, for a multi-geodesic $\gamma'=\sum \limits_{i=1}^{m} \gamma_i$, we define $E_n(\gamma'):=\sum \limits_{i=1}^{m} E_n(\gamma_i)$. The main result in the second part is:

\begin{Prop} \label{prop: n.exc}
 For $n,T>0$ we have:
 \begin{equation}
\frac{E_n(\gamma_T)}{\ell(\gamma_T)}= \frac{\sum \limits_{\gamma \in \mathcal{G}_T} E_n(\gamma)}{\sum \limits_{\gamma \in \mathcal{G}_T}\ell(\gamma)} \leq \frac{c}{n^2}, 
 \end{equation}
 for a constant $c>0$ independent of all variables and geometric data involved.
\end{Prop}

\end{itemize}

 \paragraph{Remarks}
 \begin{itemize}
 \item Lalley proved an almost equidistribution result for intersection points on compact hyperbolic surfaces \cite{lalleyEqui}. This follows from part one of the proof of Theorem \ref{theorem: main}. Our approach to the result differs from \cite{lalleyEqui} and uses the theory of geodesic currents. 
 
 \item A. Katz established an effective equidistribution result for the intersection point of closed geodesics on compact negatively curved surfaces (work in progress).
 
     \item We also prove an equidistribution result involving the directions of the geodesics at the intersection points. In other words, let $\mathcal{I}_1(X)$ be the set of all triples $(p, v_1, v_2)$, where $v_1,v_2$ are two non-parallel tangent lines at $P$. Then we prove an equidistribution result in $\mathcal{I}_1(X)$; see Corollary \ref{cor: equi.general}.
 
 \item For a finite arc $\gamma$, the first part of the proof of Theorem \ref{theorem: main} implies that the intersection points between $\gamma$ and all closed geodesics are equidistributed in $\gamma$ with respect to the length measure on $
 \gamma$ as $T \to \infty$. See $\S \ref{sec: geod.arc}$.  
 
 \end{itemize}

\paragraph{Acknowledgments}
I would like to thank C. McMullen for his continuous support and constructive suggestions, and A. Eskin, A. Wilkinson, R. Saavedra, and Y. Zhang for their helpful comments. Finally, I am grateful to the anonymous referee for a careful reading of the manuscript and for valuable suggestions.

\section{Geodesic currents}\label{sec: Pre}

This section presents some basic concepts and results regarding geodesic currents and their extensions to finite-area surfaces. For reference, see \cite{Bon.gc.3} \cite{Bon.gc}.

Let $X:=\mathbb{H}/\Gamma$ be a hyperbolic surface of finite area without boundary where $\Gamma$ is the fundamental group and $\mathbb{H}$ the upper half-plane. Each geodesic in $\mathbb{H}$ can be represented by its two boundary points at $\infty$; therefore, the space of all unoriented geodesics is homeomorphic to $G(\mathbb{H}):=(S^1 \times S^1 \backslash \Delta)/ (\mathbb{Z}/2)$. Let $\mathbb{P}(X)$ be the tangent line bundle of $X$ and $T_1(X)$ the unit tangent bundle of $X$. Geodesic flow $\phi$ gives a foliation $\mathcal{F}$ of $\mathbb{P}(X)$.

A \emph{geodesic current} can be defined and interpreted in various ways:

\begin{enumerate}
\item It is a positive $\Gamma-$invariant measure on $G(\mathbb{H})$. 
\item It is a positive $\phi-$invariant and involution invariant measure on $T_1(X)$. The involution map sends $v \in T_1(X)$ to $-v$. 

\item It is a positive transverse measure of foliation $\mathcal{F}$. In other words, it assigns a measure to each plane $V\subset \mathbb{P}(X)$ transverse to the leaves of $\mathcal{F}$ such that these measures are invariant when $V$ moves along the leaves of $\mathcal{F}$ (invariant under holonomy). See \cite[section 4.1]{Bon.gc.3}. 
\end{enumerate}
We may not mention which interpretation we use when it is clear. 

The third definition implies that geodesic currents are topological. This means that there is a canonical correspondence between the geodesic currents of two marked hyperbolic surfaces in $\mathcal{T}_g$, the Teichm\"uller space of marked complete hyperbolic surfaces with genus $g$. This correspondence sends homotopic closed geodesics to each other.

 The first and third definitions are equivalent because for each transverse plane to $\mathcal{F}$, we can consider the set of all geodesics transverse to the plane. This maps a measure on $G(\mathbb{H})$ to a transverse measure on $\mathcal{F}$. To see their connection with the second definition, note that a measure on $T_1(X)$ that satisfies the assumptions in $(\romannumeral 2)$ can be written locally as $C\times d\ell/2$ where $d\ell$ is the length measure along the geodesics and $C$ can be interpreted as a transverse measure of $\mathcal{F}$. On the other hand, the product of a transverse measure of $\mathcal{F}$ with $d\ell/2$ gives a measure on $T_1(X)$ which is $\phi-$invariant. We include the factor $1/2$ because the geodesics in $G$ and $\mathcal{F}$ are considered unoriented, whereas the geodesics in $T_1(X)$ are oriented, which would otherwise lead to double counting.
 
  Denote $C_c^*(M)$ as the dual of the set of compactly supported continuous functions on a topological space $M$.
 Let $\mathcal{C}$ be the space of all geodesic currents. Consider $\mathcal{C}$ equipped with the weak topology. That is, a sequence $C_n$ of geodesic currents converges to $C$ if 
  \begin{equation}
  \lim \limits_{n \to \infty} \int_{T_1(X)} f \, dC_n \to \int_{T_1(X)} f \, dC ,
  \end{equation}
  for every $f \in C_c^*(T_1(X))$.
 
 Closed geodesics are examples of geodesic currents; consider the half of the length measure on the unit tangent vectors to a closed geodesic. This is a $\phi$ and involution invariant measure on $T_1(X)$. Moreover, the set of weighted closed geodesics is dense in $\mathcal{C}$. The length of closed geodesics and their intersection number can be extended to geodesic currents. When $X$ is compact, we can define the length and intersection number as the continuous extension of these functions for closed geodesics. But when $X$ has cusps, the intersection number $i(.,.)$ is no longer continuous. So we define them explicitly as follows.   
 
 The \emph{length} of a geodesic current $C$ is defined as the total volume of $T_1(X)$ with respect to $C$, denoted by $\ell_X(C)$ (or simply $\ell(C)$). In this paper, we restrict $\mathcal{C}$ to geodesic currents of finite length; whenever we consider a geodesic current, we assume that it has a finite length.

Let $\mathbb{P}(X) \oplus \mathbb{P}(X)$ be the Whitney sum of the bundle $\mathbb{P}(X) \to X$, which is the space of triples $(x, \lambda_1, \lambda_2)$ where $x \in X$ and $\lambda_1, \lambda_2$ are tangent lines at $x$. 
Assume that $P_1$ and $P_2$ are projections of $\mathbb{P}(X) \oplus \mathbb{P}(X)$ to $\mathbb{P}(X)$ by forgetting the second and the third components, respectively.
Let $\mathcal{F}_i$ be the foliation on $\mathbb{P}(X) \oplus \mathbb{P}(X)$ where its leaves are the preimage of the leaves of $\mathcal{F}$ under the map $P_i$, for $i=1,2$. The space $\mathbb{P}(X) \oplus \mathbb{P}(X)$ is a $4$ dimensional manifold foliated by $\mathcal{F}_1$ and $\mathcal{F}_2$ which are transverse outside the diagonal $\triangle$. The diagonal is the set of all triples $(x,\lambda, \lambda)$. Define $\mathcal{I}_1(X):=\mathbb{P}(X) \oplus \mathbb{P}(X)\backslash \triangle$.

Every two geodesic currents $C_1$ and $C_2$ define a measure $I(C_1,C_2)$ on $\mathcal{I}_1(X)$ as follows.  
 The geodesic current $C_1$ is a transverse measure of $\mathcal{F}$ on $\mathbb{P}(X)$. Therefore, $P_1^*(C_1)$ defines a measure on each plane in $\mathcal{I}_1(X)$ transverse to $\mathcal{F}_1$. We know $\mathcal{F}_1$ and $\mathcal{F}_2$ are transverse in $\mathcal{I}_1(X)$. Therefore, $P_1^*(C_1)$ induces a measure on each leaf of $\mathcal{F}_2$. Similarly,  $P_2^*(C_2)$ induces a measure on each leaf of $\mathcal{F}_1$. Define $I(C_1,C_2)$ as the product measure $P_1^*(C_1) \times P_2^*(C_2)$. The \emph{Intersection number} $i(C_1,C_2)$ of $C_1$ and $C_2$ is the total volume of $\mathcal{I}_1(X)$ with respect to $I(C_1,C_2)$. 

Given closed geodesics $\gamma_1, \gamma_2$, we can see that $i(\gamma_1,\gamma_2)$ (as defined above) is equal to the total number of transverse intersection points between $\gamma_1$ and $\gamma_2$ counted with multiplicity. By \emph{multiplicity}, we mean that a point is counted with weight if the closed geodesics pass through it multiple times. To find the weight at a point $p$, we can homotopically move the curves in a neighborhood of $p$ to have only simple intersections. Then, the minimum number of simple intersections that can be obtained is the weight. We can see that the intersection number is a topological invariant. In other words, for closed curves $\alpha$ and $\beta$, $i(\alpha,\beta)$ is the minimal number of transverse intersection points between the representatives in the homotopy classes $[\alpha], [\beta]$ (the points are counted with multiplicity). Geodesic representatives always attain the minimum.

There is a canonical geodesic current $L_X$, called $\emph{Liouville current}$, which is equal to $L/\Gamma$; $\Gamma$ is the fundamental group and $L$ is a measure on $G(\mathbb{H})$ defined as follows. Consider the intervals $[a,b],[c,d]$ on $S^1$ that present the points at $\infty$ of the upper half-plane model. Then $L-$measure of the set of geodesics with one endpoint in $[a,b]$ and the other endpoint in $[c,d]$ is:
\begin{equation}
L([a,b]\times [c,d])=\bigg |  \log \bigg | \frac{(a-c)(b-d)}{(a-d)(b-c)} \bigg | \bigg |.
\end{equation}

We can see that $L$ is invariant under $Isom^+(\mathbb{H})=SL(2,\mathbb{R})$. 

Consider a geodesic $h$ in $\mathbb{H}$ and a point $P$ on it. These give local coordinates $(\theta,x)$ to the set of geodesics that intersect $h$, where $\theta$ is the angle a geodesic makes with $h$ and $x$ is the coordinate of its intersection with $h$ (when $P$ is considered as the origin).
We can see that $ L_X=1/2\sin(\theta)d\theta dx$ \cite{Bon.gc}.  

For any geodesic current $C$, we have $i(C, L_X)= \ell_X(C)$; in particular, we have $i(L_X, L_X)=\ell_X(L_X)=\pi^2(2g-2+n)$. Bonahon proved this for compact surfaces \cite[Prop 15]{Bon.gc}. The same proofs work for complete finite-area surfaces. We repeat the proofs below.

An $H$-box in $\mathbb{H}$ is defined as the subset of $\mathbb{P}(\mathbb{H})$ consisting of all tangent lines to geodesic arcs that connect two fixed geodesic arcs (called the edges of $H$) in $\mathbb{H}$. An $H$-box in $X$ is the projection of such an $H$-box in $\mathbb{H}$ to $X$. See Figure \ref{fig: Box}.

Given an $H$-box $b$, consider the transversal plane in $\mathbb{P}(X)$ consisting of all tangent lines in $b$ based at one of its boundary edges. Any geodesic current $C$ induces a measure on this transversal plane via its transverse measure. We define the measure of $b$ with respect to $C$ as the total mass assigned to this transversal plane.

\begin{figure}[h]

    \centering

\tikzset{every picture/.style={line width=0.75pt}} 

\begin{tikzpicture}[x=0.55pt,y=0.65pt,yscale=-1,xscale=1]

\draw    (125,87) .. controls (165,57) and (185,117) .. (225,87) ;
\draw    (197.2,145.2) .. controls (110.2,183.2) and (91.2,104.2) .. (125,87) ;
\draw    (225,87) .. controls (313.2,33.2) and (356.2,151.2) .. (262.2,156.2) ;
\draw    (197.2,145.2) .. controls (230.2,131.2) and (242.2,156.2) .. (262.2,156.2) ;
\draw    (132.2,120.2) .. controls (154.2,132.2) and (168.2,132.2) .. (190.2,120.2) ;
\draw    (237.2,111.2) .. controls (262.2,125.2) and (275.2,118.2) .. (289,107) ;
\draw    (142,124) .. controls (156.2,116.2) and (164.2,116.2) .. (180.2,126.2) ;
\draw    (250,114) .. controls (258.2,102.2) and (271.2,103.2) .. (283.2,112.2) ;
\draw   (344.95,112.25) -- (362.98,112.25) -- (362.98,105) -- (375,119.5) -- (362.98,134) -- (362.98,126.75) -- (344.95,126.75) -- cycle ;
\draw   (396.66,123.04) -- (399.78,81.13) -- (439.49,99.96) -- (436.36,141.87) -- cycle ;
\draw   (451.42,84.13) -- (453.52,35.17) -- (492.58,57.87) -- (490.48,106.83) -- cycle ;
\draw  [dash pattern={on 0.84pt off 2.51pt}]  (399.78,81.13) -- (453.52,35.17) ;
\draw  [dash pattern={on 0.84pt off 2.51pt}]  (411,101) -- (462.15,54) ;
\draw  [dash pattern={on 0.84pt off 2.51pt}]  (439.49,99.96) -- (492.58,57.87) ;
\draw  [dash pattern={on 0.84pt off 2.51pt}]  (396.66,123.04) -- (451.42,84.13) ;
\draw  [dash pattern={on 0.84pt off 2.51pt}]  (436.36,141.87) -- (490.48,106.83) ;
\draw  [dash pattern={on 0.84pt off 2.51pt}]  (430.15,120) -- (477.15,85) ;
\draw [color={rgb, 255:red, 74; green, 144; blue, 226 }  ,draw opacity=1 ]   (186.42,107.17) .. controls (195.14,112.83) and (204.14,122.83) .. (204.42,135.17) ;
\draw [color={rgb, 255:red, 74; green, 144; blue, 226 }  ,draw opacity=1 ]   (219.42,96.17) .. controls (222.14,109.83) and (229.14,117.83) .. (237.42,124.17) ;
\draw [color={rgb, 255:red, 74; green, 144; blue, 226 }  ,draw opacity=1 ]   (202.11,125.7) -- (219.81,105.21) ;
\draw [shift={(221.11,103.7)}, rotate = 130.82] [color={rgb, 255:red, 74; green, 144; blue, 226 }  ,draw opacity=1 ][line width=0.75]    (10.93,-3.29) .. controls (6.95,-1.4) and (3.31,-0.3) .. (0,0) .. controls (3.31,0.3) and (6.95,1.4) .. (10.93,3.29)   ;
\draw [color={rgb, 255:red, 74; green, 144; blue, 226 }  ,draw opacity=1 ]   (424.2,125.2) -- (463.13,63.89) ;
\draw [shift={(464.2,62.2)}, rotate = 122.41] [color={rgb, 255:red, 74; green, 144; blue, 226 }  ,draw opacity=1 ][line width=0.75]    (10.93,-3.29) .. controls (6.95,-1.4) and (3.31,-0.3) .. (0,0) .. controls (3.31,0.3) and (6.95,1.4) .. (10.93,3.29)   ;
\draw  [color={rgb, 255:red, 74; green, 144; blue, 226 }  ,draw opacity=1 ] (204.01,118.31) .. controls (206.64,118.28) and (208.79,117.73) .. (210.46,116.67) .. controls (209.27,118.25) and (208.56,120.36) .. (208.33,122.97) ;
\draw  [color={rgb, 255:red, 74; green, 144; blue, 226 }  ,draw opacity=1 ] (438.37,95.4) .. controls (442.2,93.84) and (445.17,91.77) .. (447.28,89.18) .. controls (446.03,92.28) and (445.64,95.88) .. (446.12,99.99) ;

\draw (403,162.4) node [anchor=north west][inner sep=0.75pt]    {$T_{1}( X)$};
\draw (207,163.4) node [anchor=north west][inner sep=0.75pt]    {$X$};
\draw (181,124.4) node [anchor=north west][inner sep=0.75pt]  [font=\small]  {$\alpha _{1}$};
\draw (234,125.4) node [anchor=north west][inner sep=0.75pt]  [font=\small]  {$\alpha _{2}$};
\draw (470,74.4) node [anchor=north west][inner sep=0.75pt]    {$\alpha _{2}$};
\draw (401.78,93.53) node [anchor=north west][inner sep=0.75pt]    {$\alpha _{1}$};

\end{tikzpicture}

    \caption{An $H-$Box}
    \label{fig: Box}
   
\end{figure}

\begin{Prop}\label{prop: length.LX}
Let $X$ be a hyperbolic surface with finite area. We have $i(C,L_X)=\ell_X(C)$ for any geodesic current $C$.
\end{Prop}

\begin{proof}
     We show that the Liouville measure (the measure corresponding to Liouville current) of the set of geodesics intersecting a geodesic arc $\alpha$ is equal to the length of $\alpha$. Consider the local coordinates $(x,\theta)$ defined from a fixed geodesic $\gamma_0$ with a fixed point $P_0$ on $\gamma_0$. As mentioned earlier, we have $L_X=1/2 \sin \theta d\theta dx$ in these local coordinates. Therefore, the Liouville measure of the set of all geodesics intersecting an arc $\alpha$ is equal to
     \begin{equation}
     \int \limits_{\alpha} \int \limits_{0}^{\pi} \frac{1}{2}\sin{\theta}d\theta dx= \int \limits_{\alpha} dx,
     \end{equation}
     
     which is the length of $\alpha$. 
     
      Now consider an $H-$box $B$ in $\mathbb{P}(X)$. Let $D^2$ denote the closed unit disk in $\mathbb{R}^2$. Then we have $B \cong D^2 \times [0,1]$(homeomorphic) and for any point $p \in D^2$, the set $p \times [0,1]$ corresponds to a geodesic arc between the edges of $B$. If we integrate the measure $I(C,L_X)$ along the third component $\lambda_2$ in $\mathcal{I}_1(X)$ then we obtain a measure $\mu$ on $\mathbb{P}(X)$. The measure $\mu$ on $B$ is the product measure $C\times d\ell_X$, since integrating the Liouville measure gives the length measure along the geodesic arcs $p \times [0,1]$ and the induced measure on $D^2$ is $C$. As a result, $i(C,L_X)$ which is the total volume of $\mathbb{P}(X)$ with respect to $\mu$ is equal to $\ell_X(C)$.
\end{proof}
    
For any two geodesic currents $C_1, C_2$, the push forward of the measure $I(C_1, C_2)$ by the projection $P: \mathcal{I}_1(X) \to X$ which sends $(x, \lambda_1, \lambda_2)$ to $x$,  gives a measure $P_*I(C_1,C_2)$ on $X$. We may also refer to this measure as $I(C_1,C_2)$.

\begin{Lemma}\label{lemma: push c.g}
Given closed geodesics $\gamma_1$ and $\gamma_2$, then $P_*I(\gamma_1, \gamma_2)$ is a measure supported on the intersection points of $\gamma_1$ and $\gamma_2$ and we have:
\begin{equation}
P_*I(\gamma_1, \gamma_2)= \sum_{p\in \gamma_1\cap \gamma_2} m_p \delta_p,
\end{equation}
where the sum is over the transverse intersection points of $\gamma_1$ and $\gamma_2$ and $m_p$ is the multiplicity at the intersection point $p$.
\end{Lemma}

\begin{proof}
It is easy to see that the only points in the support are the ones at the intersection of 
$\gamma_1$ and $\gamma_2$. On the other hand, for an intersection point $p$, the point $(p,\lambda_1,\lambda_2)$ is in the support of $I(\gamma_1,\gamma_2)$ if and only if $\lambda_1$ and $\lambda_2$ are tangent lines to $\gamma_1$ and $\gamma_2$, respectively. Therefore, the number of points in $\mathcal{I}_1(X)$
with the third component equal to $p$ is exactly $m_p$, the multiplicity at $p$.
\end{proof}

\begin{Prop}\label{prop: meas.I(X)} The measure $I(L_X,L_X)$ is locally equal to $1/4\sin(\theta) \, \alpha\times d\theta_1 \times d\theta_2$, where $\alpha$ is the area measure on $X$, $\theta_1$ and  $\theta_2$ are Lebesgue measure on $S^1$ corresponding to the second and third components in $\mathcal{I}_1(X)$, and $\theta$ is the angle between $\theta_1$ and $\theta_2$. 
\end{Prop}

\begin{proof}

As we explained earlier, $\mathcal{I}_1(X)$ is foliated by $\mathcal{F}_1$ and $\mathcal{F}_2$, and the transverse measures on their leaves are locally $1/2 \sin(\theta_1) \, d\theta_1 dx$ and $1/2 \sin(\theta_2) \, d\theta_2 dy$. The measure $I(L_X,L_X)$ is their product measure and is locally $1/4 \sin(\theta) \, d\theta_1 d\theta_2 \alpha$. We used the fact that when the angle between two geodesics is $\theta$, then the hyperbolic area measure is locally $\sin(\theta)\, dx dy$ where $dx,dy$ are the length measures along the two geodesics. Moreover, in this setting $\sin(\theta_1)=\sin(\theta_2)=\sin(\theta)$.   
\end{proof}

\begin{Cor} \label{cor: push Lio}
The induced measure $P_*I(L_X, L_X)$ is $\frac{\pi}{2}\alpha$ and $i(L_X,L_X)=\pi^2(2g-2+n)$.
\end{Cor}

\begin{proof}
Consider the following diagram:

\[
\xymatrix{I(X) \ar[r]^{P_2} \ar[dr]^P & \mathbb{P}(X) \ar[d]^{P'}\\  & X
}
\]
where $P'$ sends $(x, \lambda_1)$ to $x$. From Proposition \ref{prop: meas.I(X)}, we can see that the push forward of the measure $I(L_X, L_X)$ to $\mathbb{P}(X)$ by $P_2$ is the volume measure on $\mathbb{P}(X)$ which is locally equal to $1/2 \, \alpha d\theta$. Therefore, its pushforward to $X$ is $\pi/2 \, \alpha$ (since the range of $\theta$ is $[0,\pi]$), as required.
\end{proof}

\section{Intersection measure}

As we explained in \S \ref{sec: Pre}, every two geodesic currents $C_1,C_2$ define an intersection measure $I(C_1,C_2)$ on $\mathcal{I}_1(X)$. This section studies the basic properties of $I(C_1,C_2)$.

The intersection number $i(C_1,C_2)$ is the total measure of $\mathcal{I}_1(X)$ with respect to $I(C_1,C_2)$. In \cite[Prop. 4.5]{Bon.gc.3}, Bonahon proves that the intersection number (as a function on $\mathcal{C}\times \mathcal{C}$) is continuous when $X$ is compact (or when we restrict the geodesics into a compact subset of $X$). When $X$ has a finite area, the intersection number is not continuous, but we show that a weaker similar property still holds. In this section, we also prove the density of the intersection points between all closed geodesics. Although this follows directly from Theorem \ref{theorem: main}, we include a more concrete and self-contained proof here, which may be of independent interest.

Let $X$ be a complete hyperbolic surface of finite area and $\phi_t$ the geodesic flow. Let $\phi_{[0,t]}(v)$ denote the geodesic arc in $T_1(X)$ which has length $t$ and starts at $v \in T_1(X)$.

  Let $D=d\times d\theta$ be the metric on $T_1(X)$ where $d$ is the hyperbolic metric on $X$ and $d\theta$ refers to the Lebesgue measure on $S^1$. We will apply the Anosov closing lemma to prove the density result; see \cite{katokBook} for a reference to this lemma.

\begin{Lemma}(Anosov closing lemma)\label{lem: Anosov}
For any $\epsilon>0$ there exists $\delta>0$ such that if 
\begin{equation}
D(v, \phi_t(v))< \delta
\end{equation}
then there is a closed geodesic $\gamma=\{ \phi_t(w)\}_{s=0}^{t'}$ of length $t'$ where $|t-t'|<\epsilon$ and 
\begin{equation}
 D(\phi_s(v), \phi_s(w)) <\epsilon \quad \text{for each} \enspace  0\leq s \leq t .
\end{equation}

\end{Lemma}

\begin{Prop}\label{prop: dense.intpoint}
The set of all intersection points between all closed geodesics is a dense subset of $X$. 
\end{Prop}

\begin{proof}
     Consider a small open ball $B$ in $X$. We show that two closed geodesics exist that intersect in $B$. It is well known that almost all points in $T_1(X)$ are recurrent with respect to the geodesic flow. In other words, for almost every point $p \in T_1(X)$ there is a sequence $t_i \to \infty$ such that $\phi_{t_i}(p)$ tends to $p$ when $i \to \infty$. Let $v_1,v_2$ be two such vectors in $B$ on two transverse geodesic arcs in $B$ such that $D(v_1,\phi_{t_1}(v_1)), D(v_2,\phi_{t_2}(v_2))<\delta$ for some $t_1,t_2>0$. We choose $\delta>0$ small enough. If we connect $v_1$ and $\phi_{t_1}(v_1)$ by a small arc, then we obtain a closed curve and by Anosov closing lemma its geodesic representative $\gamma_1$ stays $\epsilon-$close to $\phi_{[0,t_1]}(v_1)$ at all time $t \in [0,t_1]$. Therefore, $\gamma_1$ is very close to the geodesic passed through $v_1$ inside $B$, Similarly, we can construct a closed geodesic $\gamma_2$ using $v_2$ where $\gamma_2$ is very close to $\phi_{[0,t_2]}(v_2)$. The geodesics containing $v_1,v_2$ intersect transversely in $B$, therefore, $\gamma_1$ and $\gamma_2$ intersect at a point in $B$.     
\end{proof}

\begin{Lemma}\label{lem: int.noncomp.notcon}
The intersection number $i(.,.)$ is not a continuous function when $X$ has cusp(s).
\end{Lemma}

\textit{Sketch of the proof.}
Let $N$ be a horoball neighborhood of a cusp. Consider an essential closed curve $\gamma$ that enters $N$ once, goes around the cusp once, and leaves $N$. Remove the subarc of $\gamma$ inside $N$ and instead attach a subarc in $N$ which goes around the cusp exactly $n$ times. We obtain a closed curve and refer to its geodesic representative as $\gamma_n$. See Figure \ref{fig: gamman}.
The limit of the geodesic currents $\gamma_n'$s is a geodesic ray $r$ whose ends go to $\infty$ (in the cusp) as $n \to \infty$. We know that $i(\gamma_n,\gamma_n)\sim 2n$ and $i(r,r) \sim c_0$, for a constant $c_0$ depending on $\gamma$. This shows that $i(.,.)$ is not continuous, as required.

\begin{figure}[H]
    \centering

\tikzset{every picture/.style={line width=0.75pt}} 

\begin{tikzpicture}[x=0.60pt,y=0.65pt,yscale=-1,xscale=1]

\draw    (248.2,120.2) .. controls (178.2,147.2) and (208.2,232.2) .. (306.2,220.2) ;
\draw    (306.2,220.2) .. controls (365.2,207.2) and (386.2,159.2) .. (527.2,132.2) ;
\draw    (248.2,120.2) .. controls (303.2,105.2) and (426.2,154.2) .. (520.2,124.2) ;
\draw    (254.2,168.2) .. controls (275.2,156.2) and (295.2,157.2) .. (318.2,167.2) ;
\draw    (265,162) .. controls (281.2,184.2) and (298.2,176.2) .. (306.2,163.2) ;
\draw    (236.2,174.2) .. controls (241.2,197.8) and (294,218) .. (334,188) ;
\draw    (236.2,174.2) .. controls (224.2,104.8) and (342.2,148.2) .. (354.2,153.2) ;
\draw    (354.2,153.2) .. controls (372.2,160.2) and (383.2,180.2) .. (403.2,171.2) ;
\draw  [dash pattern={on 0.84pt off 2.51pt}]  (403.2,171.2) .. controls (423.73,156.26) and (413.88,135.08) .. (423.65,133.6) ;
\draw    (434.88,159.08) .. controls (437.87,149.6) and (436.87,144.6) .. (444.87,139.6) ;
\draw    (445.87,148.6) .. controls (428.51,149.99) and (435.53,134.47) .. (423.65,133.6) ;
\draw    (334,188) .. controls (374,158) and (363.13,155.85) .. (384.13,129.85) ;
\draw  [dash pattern={on 0.84pt off 2.51pt}]  (384.13,129.85) .. controls (424.65,133.6) and (407.88,160.08) .. (434.88,159.08) ;
\draw  [dash pattern={on 0.84pt off 2.51pt}]  (466.67,133.04) .. controls (481.32,134.43) and (476.02,143.73) .. (487.67,141.04) ;
\draw  [dash pattern={on 0.84pt off 2.51pt}]  (464.67,148.04) .. controls (485.47,137.86) and (479.68,132.9) .. (486.17,131.45) ;
\draw    (486.17,131.45) .. controls (498.09,132.22) and (494.23,138.79) .. (487.67,141.04) ;
\draw    (456.93,143.27) .. controls (461,137.93) and (462.93,135.27) .. (466.67,133.04) ;
\draw    (452.73,138.2) .. controls (463.73,140.2) and (459.73,144.2) .. (464.67,148.04) ;

\draw (428,108.4) node [anchor=north west][inner sep=0.75pt]    {$\gamma _{n}$};
\draw (439.05,144.11) node [anchor=north west][inner sep=0.75pt]  [font=\tiny,rotate=-340.94,xslant=0.01]  {$\dotsc $};

\end{tikzpicture}

    \caption{An example that shows the intersection number is not continuous}
    \label{fig: gamman}
\end{figure}

In the compact case, the continuity of the intersection number relies on the fact that if two geodesic currents assign a small mass to $H-boxes$, then their intersection number is correspondingly small. This control breaks down in the presence of cusps, as seen in the construction in Lemma~\ref{lem: int.noncomp.notcon}, where a geodesic current may have a small mass in the cusp region while still having a large number of self-intersections. More precisely, we have:

\begin{Ex}\label{ex: not.cont.int}
    Consider the geodesic current $\gamma_n/\ell(\gamma_n)$ as described in Lemma \ref{lem: int.noncomp.notcon}. Its measure inside an $\epsilon-$neighborhood of the cusp is small, but its intersection number in that neighborhood is 
    
\begin{equation}
    \sim \frac{2n}{\ell(\gamma_n)^2} \sim \frac{2n}{c+2\log n}
    \end{equation}
\end{Ex}
which goes to $\infty$ when $n \to \infty$.

Here is the result weaker than the continuity of $i:\mathbb{C}\times \mathbb{C} \to \mathbb{R}$.

\begin{Thm}\label{thm: cont.I(.,.)}
    Assume that the geodesic currents $B_i$ and $C_i$ converge to $B,C\in \mathcal{C}$ in $C_c^*(T_1(X))$, as $i \to \infty$, respectively. Then as $i \to \infty$, we have
    \begin{equation}
    \int_{\mathcal{I}_1(X)} f \, dI(B_i,C_i) \to \int_{\mathcal{I}_1(X)} f \, dI(B,C),
    \end{equation}
    for any bounded continuous function $f$ supported on $K \cap  \mathcal{I}_1(X)$ for some compact subset $K \subset \mathbb{P}(X)\oplus \mathbb{P}(X)$.
\end{Thm}

Note that it is mandatory to restrict the convergence in Theorem \ref{thm: cont.I(.,.)} to the functions whose supports are disjoint from a cusps' neighborhood (being in $K \cap \mathcal{I}_1(X)$ where $K$ is a compact subset of $\mathbb{P}(X)\oplus \mathbb{P}(X)$). For instance, this theorem is not necessarily true for $f=1$; the total measures $I(B_i,C_i)(\mathcal{I}_1(X))=i(B_i,C_i)$ do not necessarily converge to the total measure $I(B,C)(\mathcal{I}_1(X))=i(B,C)$ since $i(-,-)$ is not a continuous function (see Lemma \ref{lem: int.noncomp.notcon}). The following results are the heart of the proof of Theorem \ref{thm: cont.I(.,.)}.

Let $\gamma$ be a simple closed geodesic. Define the \textit{winding number} of a geodesic arc $\alpha$, whose endpoints lie on opposite boundaries of the collar of $\gamma$, as the number of times $\alpha$ wraps around $\gamma$. More precisely, consider the orthogonal projection of $\alpha$ onto $\gamma$, which traces out a curve of length $\ell$. We define the winding number of $\alpha$ around $\gamma$ as $ \lfloor \ell/\ell(\alpha) \rfloor$.

\begin{Lemma}\label{lemma: int.arc.collar}
Consider two geodesic arcs in a collar neighborhood of a simple closed geodesic $\gamma$ whose endpoints lie on the opposite boundaries of the collar, and they have the winding numbers $n,m$ around $\gamma$. Then their intersection number is $\leq m+n+2$.  
\end{Lemma}

\begin{proof} Geodesic arcs minimize intersection number among all curves in their homotopy classes with fixed endpoints. Therefore, it is enough to establish the desired bound for some representatives of these classes.

Geodesic $\gamma$ splits the collar neighborhood into two components, denoted $R_1$ and $R_2$. Let $\alpha$ and $\beta$ be geodesic arcs with winding numbers $n$ and $m$, respectively. 

We now define new curves $\alpha'$ and $\beta'$, each homotopic to $\alpha$ and $\beta$ (with fixed endpoints), as follows. The curve $\alpha'$ goes around $\gamma$ for $n$ times within $R_1$, and in $R_2$, it is a geodesic segment orthogonal to $\gamma$, thus having zero winding number there. Similarly, we define $\beta'$ that goes around $
\gamma$ for $m$ times within $R_2$ and it is an orthogonal geodesic to $\gamma$ in $R_1$. See Figure \ref{fig: winding}.

 Now we can see that, in this configuration, $\alpha'$ and $\beta'$ intersect at most $n+1$ times in $R_1$ and at most $m+1$ times in $R_2$. Therefore, we have $i(\alpha,\beta) \leq i(\alpha',\beta')\leq n+m+2$, as required. 

\begin{figure}[H]
    \centering

\tikzset{every picture/.style={line width=0.75pt}} 

\begin{tikzpicture}[x=0.75pt,y=0.75pt,yscale=-0.6,xscale=0.5]

\draw    (168,104) .. controls (298.2,179.2) and (410.2,156.2) .. (511.2,90.2) ;
\draw    (167,253) .. controls (273.2,177.2) and (432.2,210.2) .. (536.2,254.2) ;
\draw    (168,104) .. controls (136.2,117.2) and (137.2,248.2) .. (167,253) ;
\draw    (168,104) .. controls (199.2,120.2) and (198.2,236.2) .. (167,253) ;
\draw    (511.2,90.2) .. controls (566.2,85.2) and (572.2,246.2) .. (536.2,254.2) ;
\draw  [dash pattern={on 0.84pt off 2.51pt}]  (511.2,90.2) .. controls (500.2,130.2) and (502.2,211.2) .. (536.2,254.2) ;
\draw    (335.2,150.2) .. controls (356.4,149.4) and (362.2,200.2) .. (345.2,208.2) ;
\draw  [dash pattern={on 0.84pt off 2.51pt}]  (335.2,150.2) .. controls (324.2,156.2) and (322.2,204.2) .. (345.2,208.2) ;
\draw [color={rgb, 255:red, 74; green, 144; blue, 226 }  ,draw opacity=1 ]   (188,216) .. controls (222.2,179.2) and (211.53,131.52) .. (224.53,131.52) ;
\draw [color={rgb, 255:red, 74; green, 144; blue, 226 }  ,draw opacity=1 ] [dash pattern={on 0.84pt off 2.51pt}]  (224.53,131.52) .. controls (263.53,139.52) and (255.2,214.2) .. (245.2,216.2) ;
\draw [color={rgb, 255:red, 74; green, 144; blue, 226 }  ,draw opacity=1 ]   (245.2,216.2) .. controls (231.2,216) and (245.2,141) .. (274.2,145) ;
\draw [color={rgb, 255:red, 74; green, 144; blue, 226 }  ,draw opacity=1 ] [dash pattern={on 0.84pt off 2.51pt}]  (274.2,145) .. controls (316.2,152) and (311.2,205) .. (304.2,208) ;
\draw [color={rgb, 255:red, 74; green, 144; blue, 226 }  ,draw opacity=1 ]   (304.2,208) .. controls (268.2,202) and (298.2,160) .. (352.2,172) ;
\draw [color={rgb, 255:red, 74; green, 144; blue, 226 }  ,draw opacity=1 ]   (352.2,172) .. controls (413.2,169) and (497.2,158) .. (550.2,137) ;
\draw [color={rgb, 255:red, 208; green, 2; blue, 27 }  ,draw opacity=1 ]   (193,191) .. controls (251.2,180) and (318.2,189) .. (354.2,191) ;
\draw [color={rgb, 255:red, 208; green, 2; blue, 27 }  ,draw opacity=1 ]   (354.2,191) .. controls (418.2,189) and (419.4,222.8) .. (441.4,221.8) ;
\draw [color={rgb, 255:red, 208; green, 2; blue, 27 }  ,draw opacity=1 ] [dash pattern={on 0.84pt off 2.51pt}]  (441.4,221.8) .. controls (459.01,225.32) and (457.33,202.57) .. (454.24,177.7) .. controls (451.19,153.12) and (446.77,126.47) .. (458.2,121) ;
\draw [color={rgb, 255:red, 208; green, 2; blue, 27 }  ,draw opacity=1 ]   (458.2,121) .. controls (483.2,110) and (496.2,239) .. (484.2,235) ;
\draw [color={rgb, 255:red, 208; green, 2; blue, 27 }  ,draw opacity=1 ] [dash pattern={on 0.84pt off 2.51pt}]  (484.2,235) .. controls (479.2,232) and (453.2,127) .. (487.2,106) ;
\draw [color={rgb, 255:red, 208; green, 2; blue, 27 }  ,draw opacity=1 ]   (487.2,106) .. controls (522.2,97) and (513.2,263) .. (553.2,233) ;

\draw (353,154.4) node [anchor=north west][inner sep=0.75pt]  [font=\small]  {$\gamma $};
\draw (190,138.4) node [anchor=north west][inner sep=0.75pt]    {$\textcolor[rgb]{0.29,0.56,0.89}{\alpha' }$};
\draw (402,180.4) node [anchor=north west][inner sep=0.75pt]    {$\textcolor[rgb]{0.82,0.01,0.11}{\beta'}$};
\draw (200,202.4) node [anchor=north west][inner sep=0.75pt]    {$R_{1}$};
\draw (475,205.4) node [anchor=north west][inner sep=0.75pt]    {$ \begin{array}{l}
R_{2}\\
\end{array}$};

\end{tikzpicture}

    \caption{A homotopic modification of $\alpha$ and $\beta$. The curves $\alpha'$ and $\beta'$ coincide with geodesics orthogonal to $\gamma$ within the regions $R_2$ and $R_1$, respectively.}
    \label{fig: winding}
\end{figure}
\end{proof}

 \begin{Prop}\label{prop: neighbor.diagonal}
Let $K$ be a compact subset of $\mathbb{P}(X)\oplus\mathbb{P}(X)$. Given geodesic currents $B$ and $C$ and $\epsilon>0$, there are neighborhoods $U_{\epsilon}, V_{\epsilon}$ of $B$ and $C$ in $\mathcal{C}$, respectively, and a neighbourhood $O_{\epsilon}\subset \mathcal{I}_1(X)$ of the diagonal $\triangle$ such that 
\begin{equation}
I(B',C')(O_{\epsilon}\cap K)<\epsilon,
\end{equation}
for all $B' \in U_{\epsilon}$ and $C' \in V_{\epsilon}$.
\end{Prop}
\begin{proof}
  The projection of $K$ to $X$ is a compact subset of $X$. Let $K'$ be the set of all tangent lines to this compact subset of $X$. Then $K'$ is a compact subset of $\mathbb{P}(X)$. Consider a cover of $K'$ with compact $H-boxes$ $K_1,\dots,K_t$, chosen sufficiently small so that each leaf within these boxes is a simple geodesic arc, and any two leaves in the same $H-box$ intersect at most once. By subdividing the edges of each $K_s$, we can split it into smaller compact $H-boxes$ with disjoint interiors, denoted $b_1,\dots,b_m$ (for simplicity, we remove the dependence of the indices on $s$). Each box $b_i$ is chosen to be of one of the following types:
  \begin{enumerate}
      \item $B(b_i)<\delta$, or 
      \item $C(b_i)<\delta$, or
      \item  there is a closed geodesic $\gamma$ with transverse measure $\geq \delta$ with respect to both $B$ and $C$ (an atom for both $B$ and $C$) which passes through $b_i$ once. Moreover, $B(b_i \backslash \gamma)$ and  $C(b_i \backslash \gamma)$ are  $<\delta'$.
  \end{enumerate}

  We will choose $\delta, \delta'>0$ small enough later. 

  Here, by $A(b)$, for a geodesic current $A$ and an $H-box$ $b$, we mean the transverse measure of $b$ with respect to $A$. The construction of the subdivision implies that $A(b_1)+\dots+A(b_m)=A(K_s)$, provided that the boundaries of the boxes $b_i$ have zero $A-$measure.

The reason we can construct such a covering of $K_s$ by $H-boxes$ is as follows. Since $B$ and $C$ are geodesic currents with finite total mass, they have at most finitely many atoms of transverse measure $\geq \delta$, named $\delta-$atom, each supported on a closed geodesic. We begin by splitting the compact set $K_s$ into sufficiently small compact $H-boxes$ so that each box contains at most one $\delta-$atomic geodesic of $B$ or $C$. For any $H-box$ that contains a $\delta-$atom for both $B$ and $C$, we can further subdivide it so that the box containing the $\delta-$atomic geodesic has the complement of measure $<\delta'$ with respect to both $B$ and $C$. Next, for any $H-box$ that does not contain a $\delta-$atom of $B$ (or of $C$), we subdivide it as needed into smaller $H-boxes$, each of which has measure $<\delta$ with respect to $B$ (or $C$, respectively). In this way, we obtain a finite splitting of $K_s$ by $H-boxes$, each of which satisfies one of the three desired types.
  
  For each $K_s$, define 
  $$
  O^s_{\epsilon}=(b_1\oplus b_1 \cup \dots \cup b_m \oplus b_m)\cap \mathcal{I}_1(X),
  $$ 
  which is a neighborhood of the diagonal $\Delta$. Let 
  $$
  O_{\epsilon}:=O^1_{\epsilon}\cup \dots \cup O^t_{\epsilon}.
  $$
  Now, our goal is to prove that $O_{\epsilon}$ has small total mass with respect to the intersection measure of $B$ and $C$, and any pair of geodesic currents in a small neighborhood of $B$ and $C$. It suffices to prove this for each individual $O^s_{\epsilon}$. 
  
   Choose small neighborhoods $U_{\epsilon}$ and $V_{\epsilon}$ of $B$ and $C$ in $\mathcal{C}$, respectively, such that for all indices $s$ we have:
   \begin{itemize}
       \item Constraint $1$: $B'(K_s)<B(K_s)+\delta, \, \, \, \, C'(K_s)<C(K_s)+\delta$
\item Constraint $2$: $B'(b_i)<B(b_i)+\delta, \, \, \, \, C'(b_i)<C(b_i)+\delta$ for $i=1, \dots, m$
       \end{itemize}
       
   for every $B' \in U_{\epsilon}$ and $C' \in V_{\epsilon}$. We will set the third constraint later and then we will explain why these constraints are achievable.
   
  Note that for any geodesic currents $A_1,A_2$, and any $H-boxes$ $e_1,e_2$, we have $I(A_1,A_2)(e_1\oplus e_2)\leq A_1(e_1)\cdot A_2(e_2)$.
  
   For $H-boxes$ of type $1$ and $2$ we have either $B'(b_i)\leq B(b_i)+\delta<2\delta$ or $C'(b_i)\leq 2\delta$. Therefore, we can see that the sum of the contribution of the sets $(b_i \oplus b_i) \cap \mathcal{I}_1(X)$, for $H-$boxes $b_i$ of type $1$ and $2$, to $I(B',C')(O^s_{\epsilon} \cap K)$ is 
  \begin{align}  
  <2\delta \sum_{j} B'(b_j)+ 2\delta \sum_j C'(b_j) \leq 2\delta(B'(K_s)+C'(K_s)) \nonumber \\<2\delta(B(K_s)+C(K_s)+2\delta).
  \end{align}
  
  Choosing $\delta>0$ small enough makes this contribution $<\epsilon/2$. Note that the cover of compact set $K'$ by $H-boxes$ $K_s$ (and $t$, the number of these boxes) is fixed and does not depend on the parameters $\delta$ and $\delta'$.
  
  We now show that the contribution of boxes $b_i \oplus b_i$ is also small for all boxes $b_i$ of type $3$. Thus, from now on, only boxes $b_i$ of the third type are considered.
  
  Assume that the closed geodesic $\gamma \subset \mathbb{P}(X)$ is the common atom of $B$ and $C$ in $b_i$. Consider $\widetilde{X}=\mathbb{H}/\Gamma_{\gamma}$ where $\Gamma_{\gamma}$ is a subgroup of $\Gamma$ generated by a hyperbolic transformation corresponding to $\gamma$. This is a covering of $X$ homeomorphic to a cylinder with $\gamma$ as its core curve. Let $\widetilde{b_i}$ be a lift of $b_i$ to $\widetilde{X}$ (the lift that intersects $\gamma$). 
  Define $G_p^i \subset \widetilde{b_i}$ as the set of all tangent lines to geodesics that pass through $b_i$ exactly $p$ times. Note that the tangent lines to $\gamma$ do not belong to any of the sets $G_p^i$ for finite $p$, since $\gamma$ passes through $\widetilde{b_i}$ infinitely many times. In fact, all such tangent lines are contained in $G_{\infty}^i$.

  According to Lemma~\ref{lemma: int.arc.collar}, the number of intersection points within $\widetilde{b_i}$ between two geodesics in $G_p^i$ and $G_q^i$ is at most $p+q+2$. Although a priori two geodesics that pass through $\widetilde{b_i}$ exactly $p$ and $q$ times, respectively, could intersect up to $pq$ times, the lemma shows that their actual number of intersections is bounded above by $p+q+2$. 

Each geodesic current on $X$ admits a natural lift to a geodesic current on $\widetilde{X}$. We can decompose the set $G_p^i$ into $p$ disjoint subsets, each corresponding to one of the $p$ distinct times a geodesic enters $\widetilde{b_i}$. Since the total transverse measure of $G_p^i$ with respect to $B'$ is $B'(G_p^i)$, each of these subsets has measure $B'(G_p^i)/p$ with respect to $B'$. Similarly, we decompose $G_q^i$ into $q$ disjoint subsets, each with transverse measure $C'(G_q^i)/q$ with respect to $C'$.

  As a result, the contribution of the intersection points between the geodesics in $G_p^i$ and $G_q^i$ to $I(B',C')$ is 
  
  \begin{equation}
  I(B',C')(G_p^i \oplus G_q^i)\leq (p+q+2)\frac{B'(G_p^i)}{p}\frac{C'(G_q^i)}{q}.
  \end{equation}  
  
  Let $k \in \mathbb{N}$ be a fixed sufficiently large constant. We assume that the neighborhoods $U_{\epsilon}$ and $V_{\epsilon}$ are small enough such that for all indices $s$ we have
\begin{itemize}
    \item Constraint $3$: $B'(\cup_{p=1}^k G_p^i)<B(\cup_{p=1}^k G_p^i)+\delta,$ and \\
    $C'(\cup_{p=1}^k G_p^i)< C(\cup_{p=1}^k G_p^i)+\delta,$
\end{itemize}
  for every $B' \in U_{\epsilon}$ and $C'\in V_{\epsilon}$.
   
   These constraints on the neighborhoods $U_{\epsilon}, V_{\epsilon}$ are achievable because the boundaries of the sets $G_p^i$ lie within $\partial b_i$, and we can choose the boxes $b_i$ and $K_s$ so that their boundaries avoid all closed geodesics, since 
$X$ has only countably many closed geodesics. Consequently, the boundaries of the sets $G_p^i, b_i, K_s$ avoid the atoms of the geodesic currents $B$ and $C$, which are supported on closed geodesics. Thus, by \cite[Lemma 4.3]{Bon.gc.3}, the boundaries $\partial \cup_{p=1}^k G_p^i, \partial b_i, \partial K_s$ have zero measure with respect to any geodesic current, ensuring the continuity of the measures under perturbation.

Moreover, from Constraint $3$, we can conclude that $B'(\cup_{p=1}^{k} G_p^i), C'(\cup_{p=1}^{k} G_p^i) \leq \delta'+\delta$, since $\cup_{p=1}^{k} G_p^i$ is a subset of $b_i\backslash \gamma$ that has measure $<\delta'$ with respect to both $B$ and $C$ (this is the constraint for type $3$ $H-$boxes).
  
  The contribution of $b_i \oplus b_i$ to $I(B',C')$ is:
  
   \begin{align}
  \sum_{1 \leq p,q \leq \infty} \textit{intersection number arises from } G_p^i, G_q^i \nonumber \\
  =\sum_{p\leq k} B'(G_p^i)C'(b_i)+ \sum_{q \leq k} C'(G_q^i)B'(b_i)\nonumber \\
  +\sum_{k<p,q<\infty} \frac{p+q+2}{pq} B'(G_p^i)C'(G_q^i)
  + \sum_{k<p, q=\infty} \frac{B'(G_p^i)}{p}C'(\gamma)\nonumber \\ + \sum_{k<q,p=\infty} \frac{C'(G_q^i)}{q}B'(\gamma).\label{inequ: lemma.diagonal.int.small}     
\end{align}  
As we explained, from Constraint $3$ we conclude that the first two terms in (\ref{inequ: lemma.diagonal.int.small}) are $\leq (\delta+\delta')(C'(b_i)+B'(b_i))$. On the other hand, for $p,q>k$ we have:
  \begin{align}\label{ineq: bound.p.q.k}
  \frac{p+q+2}{pq}<\frac{3}{k}, \, \, \, \, \frac{B'(G_p^i)}{p} \leq \frac{B'(b_i)}{k}, \, \, \, \, \frac{C'(G_q^i)}{q} \leq \frac{C'(b_i)}{k}.
  \end{align}
  From (\ref{ineq: bound.p.q.k}), we can control the other terms in (\ref{inequ: lemma.diagonal.int.small}) and we can see that the contribution of $b_i \oplus b_i$ to $I(B',C')$ is 
  
  \begin{align} 
      \leq (\delta+\delta')(C'(b_i)+B'(b_i))+ \frac{5}{k} B'(b_i)C'(b_i) 
      \label{equ: bnd.k}
  \end{align}
  Therefore, the total contribution of all $b_i \oplus b_i$ corresponding to the type 3 boxes to $I(B', C')$ is

\begin{align}
\leq (\delta+\delta')(\sum_i C'(b_i)+B'(b_i))+\frac{5}{k}\sum_i B'(b_i)C'(b_i) \nonumber \\
\leq (\delta+\delta')( C'(K_s)+B'(K_s))+\frac{5}{k} B'(K_s)C'(K_s) \nonumber \\
\leq (\delta+\delta')( C(K_s)+B(K_s)+2\delta)+\frac{5}{k}( B(K_s)+\delta)(C(K_s)+\delta). 
\end{align} 
Choosing $\delta'$ and $\delta$ small enough and $k$ large enough makes this contribution $<\epsilon/2$, thus we have $I(C',B')(O^s_{\epsilon}\cap K)<\epsilon$, as required.
\end{proof}

\emph{Proof of Theorem \ref{thm: cont.I(.,.)}.} Given a function $f$ and a compact set $K\subset \mathbb{P}(X)\oplus\mathbb{P}(X)$ as described in the statement of the theorem, for any $\epsilon>0$, consider $U_{\epsilon}, V_{\epsilon}, O_{\epsilon}$ from Proposition \ref{prop: neighbor.diagonal}. Assume that $f$ is bounded above by $b$. Then from Proposition \ref{prop: neighbor.diagonal} we have

\begin{equation}
|\int_{O_{\epsilon}\cap K}f \, dI(B_n,C_n)-\int_{O_{\epsilon}\cap K}f\,dI(B,C)| \leq 2b\epsilon.
\end{equation}

Moreover, the convergence of $B_n$'s and $C_n$'s implies that 

\begin{equation}
|\int_{E_{\epsilon}}f \, dI(B_n,C_n)-\int_{E_{\epsilon}}f\,dI(B,C)| \leq \epsilon, 
\end{equation}

when $n$ is large enough, for the compact set $E_{\epsilon}=K\backslash O_{\epsilon}$.
Combining these two inequalities implies 
\begin{equation}
\int_{\mathcal{I}_1(X)}f \, dI(B_n,C_n) \to \int_{\mathcal{I}_1(X)}f\,dI(B,C),
\end{equation}

when $n \to \infty$, as required. 
\qed

 \section{Excursion}

  In this section, we find an upper bound on the number of self-intersection points of $\gamma_T$ in a neighborhood of the cusps. As we explained before, bounding this number is necessary to complete the second part of the proof and show that there is no escape of mass to infinity in the converging sequence in Theorem \ref{thm: weaker.limit}.
  
  Let $p$ be a cusp of $X$ and $N_p(r)$ the horoball neighborhood of $p$ with horocycle boundary of length $r$, in other words, $\ell(\partial N_p(r))=r$.

  Let the \textit{winding number} of a geodesic arc $\alpha$, whose endpoints lie on $\partial N_p(r)$, be the number of times $\alpha$ wraps around the cusp. More precisely, consider the orthogonal projection of $\alpha$ onto $\partial N_p(r)$, which traces out a curve of length $\ell$. The winding number of $\alpha$ around the cusp is then defined as $ \lfloor \ell/r \rfloor$.

  Define an \emph{$n-$excursion} (or \emph{$n-$exc}) to be a geodesic arc whose endpoints are on $\partial N_p(r)$ and its winding number around the cusp is $n$.
  
  The measure $P_*I(\gamma_T,\gamma_T)(N_p(r))$ is the number of self-intersection points of $\gamma_T$ inside $N_p(r)$. We have:

  \begin{Prop}\label{prop: num.int.cusp}
 There exists a constant $c_r>0$ such that the contribution to the self-intersection number $i(\gamma_T,\gamma_T)$ from the points inside $N_p(r)$ is bounded above by $c_re^{2T}$. In other words, we have
 \begin{equation}
 P_*I(\gamma_T,\gamma_T)(N_P(r)) \leq c_r e^{2T},
 \end{equation}
 for a constant $c_r$ which tends to $0$ as $r \to 0$.
\end{Prop}

In order to prove Proposition \ref{prop: num.int.cusp}, we use the following results. Let $B_r(p)$ denote the disk of radius $r>0$ centered at a point $p$. 
  
  \begin{Lemma} \label{lemma: mrk.pnt}
  The number of geodesics with endpoints at a point $Q \in X$, and length $\leq T$ is $\leq Ce^T$ where $C>0$ is a constant only depending on $Q$.
  \end{Lemma}
  \begin{proof}
  Let $Q_0 \in \mathbb{H}$ be a preimage of $Q\in X$. Geodesics with endpoints at $Q$ and length $\leq T$ correspond to the $\pi_1$ (fundamental group) orbit of $Q_0$ in the disk of radius $T$.
  Consider the preimages (in $\mathbb{H}$) of a small embedded disk of radius $\epsilon$ around $Q$. These are separated disks inside a disk of radius $T+2\epsilon$ centered at $Q_0$. Therefore, the number of these orbit points is $<vol(B_{T+2\epsilon}(Q_0))/vol(B_{\epsilon}(Q_0))<Ce^T$ for a constant $C>0$ depending on $\epsilon$ (which depends only on $Q$).   
  \end{proof}

  \begin{Lemma}\label{lem: n,m,exc.int}
  The intersection number between an $n-$exc and an $m-$exc is $\leq 2\min(n,m)+2$.
      
  \end{Lemma}
  \begin{proof}
     Consider the universal cover $A=(-\infty, \infty) \times [1,\infty)$ for the cuspidal neighborhood with a deck transformation $(x,y) \to (x+1,y)$. A preimage of a $k-$exc is an arc in $A$ whose endpoints are on the line $y=1$ with distance $d \in [k,k+1)$ from each other. Geodesic arcs attain the minimum intersection number among the homotopic curves with fixed endpoints. Therefore, to bound the intersection number between an $n-$exc and an $m-$exc, it is sufficient to bound the number of the intersection points between two representative curves in their homotopy classes (with fixed endpoints). Without loss of generality, assume that $n \leq m$. Consider a half-circle homotopic to the $n-$exc that is inside $(-\infty,\infty)\times[1,n+2] \subset A$ and an arc homotopic to the $m-$exc which is constructed from three subarcs as follows. The middle subarc is a segment inside $(-\infty,\infty)\times\{ n+3\}$ and the two ends are vertical lines from $y=1$ to $y=n+3$. Then, we can see that the image of each vertical line intersects the image of the arc homotopic to the $n-$exc in $\leq n+1$ points. Therefore, the intersection number is $\leq 2n+2$, as required.  
  \end{proof}
  
  \begin{Lemma}\label{lem: excursion.length}
  
  Let $\eta_n$ be an $n-$exc in $N_p(r)$. Then we have
  \begin{equation}
  \ell(\eta_n) \geq 2\log n+c
  \end{equation}
  for a constant $c \in \mathbb{R}$ only depending on $r$.
  \end{Lemma}
  \begin{proof}
  Consider a semicircle and a horizontal line $y=1$ as preimages of $\eta_n$ and $\partial N_p(r) $ in $\mathbb{H}$, the upper half-plane, respectively. See Figure \ref{fig: E_n}. Let $z_1, z_2$ be the endpoints of $\eta_n$ on $\partial N_p(r) $ in $\mathbb{H}$. The points $z_1,z_2$ have the same height and the difference between their $x-coordinate$ is $\geq nr$. 
  
  \begin{figure}[h]
      \centering

\tikzset{every picture/.style={line width=0.75pt}} 

\begin{tikzpicture}[x=0.55pt,y=0.55pt,yscale=-1,xscale=1]

\draw    (156.2,175.2) -- (466.2,176.2) ;
\draw  [dash pattern={on 4.5pt off 4.5pt}]  (199,108) -- (388.2,112.2) ;
\draw  [draw opacity=0] (202.55,174.84) .. controls (202.54,174.23) and (202.54,173.63) .. (202.55,173.02) .. controls (203.23,124.74) and (245.07,86.18) .. (295.99,86.9) .. controls (346.65,87.61) and (387.22,126.93) .. (386.97,174.87) -- (294.76,174.32) -- cycle ; \draw   (202.55,174.84) .. controls (202.54,174.23) and (202.54,173.63) .. (202.55,173.02) .. controls (203.23,124.74) and (245.07,86.18) .. (295.99,86.9) .. controls (346.65,87.61) and (387.22,126.93) .. (386.97,174.87) ;  
\draw  [dash pattern={on 0.84pt off 2.51pt}]  (242.2,57.2) -- (244,175) ;
\draw  [dash pattern={on 0.84pt off 2.51pt}]  (260.2,58.2) -- (262,176) ;
\draw  [dash pattern={on 0.84pt off 2.51pt}]  (277.2,58.2) -- (279,176) ;
\draw  [dash pattern={on 0.84pt off 2.51pt}]  (345.2,58.2) -- (347,176) ;
\draw  [dash pattern={on 0.84pt off 2.51pt}]  (330.2,58.2) -- (332,176) ;
\draw [line width=1.5]  [dash pattern={on 1.69pt off 2.76pt}]  (296.1,67.1) -- (317.2,67) ;
\draw  [line width=0.75] [line join = round][line cap = round] (223.67,56.49) .. controls (223.67,44.47) and (249.42,47.82) .. (257.67,47.49) .. controls (258.64,47.45) and (263.67,43.48) .. (263.67,43.49) .. controls (265.17,47.98) and (266.72,50.4) .. (273.67,50.49) .. controls (297.33,50.82) and (321.01,50.16) .. (344.67,50.49) .. controls (350.61,50.57) and (352.67,53.15) .. (352.67,58.49) ;
\draw  [line width=2.25] [line join = round][line cap = round] (235.37,107.93) .. controls (234.46,107.93) and (231.82,108.93) .. (235.37,108.93) ;
\draw  [line width=2.25] [line join = round][line cap = round] (357.37,111.93) .. controls (357.37,109.57) and (359.72,111.93) .. (357.37,111.93) ;

\draw (179,48.4) node [anchor=north west][inner sep=0.75pt]    {$\mathbb{H}$};

\draw (390,103) node [anchor=north west][inner sep=0.75pt]    {$\partial N_p(r)$};

\draw (257,21.4) node [anchor=north west][inner sep=0.75pt]    {$n$};
\draw (286,69.4) node [anchor=north west][inner sep=0.75pt]  [font=\footnotesize]  {$ \begin{array}{l}
\eta_{n}\\
\end{array}$};
\draw (227,92.4) node [anchor=north west][inner sep=0.75pt]  [font=\scriptsize]  {$z_{1}$};
\draw (356,95.4) node [anchor=north west][inner sep=0.75pt]  [font=\scriptsize]  {$z_{2}$};

\end{tikzpicture}
  \caption{The preimage of an $n-$excursion in the upper half-plane}
      \label{fig: E_n}
  \end{figure}

  From the formula of hyperbolic length, \cite[Equ. 1.1.2]{Bsr}, we have
  \begin{align}
  \ell(\eta_n)=d(z_1,z_2)=\cosh^{-1}(1+\frac{|z_1-z_2|^2}{Im z_1 Im z_2}) \geq \log (1+bn^2)\nonumber \\
  \geq 2\log n+\log b,
  \end{align}
  where $b$ is a constant only depending on $r$. 
  \end{proof}
  Let $A_n^p(T)$ denote the number of all $n-$exc's in $N_p(1)$ among the closed geodesics of length $\leq T$. 
  
  \begin{Prop}\label{prop: exc.number}
  There is a constant $C>0$ such that
  \begin{equation}
  A_n^p(T)\leq \frac{Ce^T}{n^2},
  \end{equation}
  for any $n \in \mathbb{N}$.
  \end{Prop}
  
  \begin{proof}
  From now on, by excursion, we mean an excursion whose endpoints are on the horocycle of length $1$ of a cusp $p$. Let $S_T^n$ be the set of pairs $(\gamma, \eta_n)$ where $\gamma$ is a closed geodesic with length $\leq T$ and $\eta_n$ an $n-$exc of $\gamma$ in $N_p(1)$. Let $G_T(Q)$ be the set of all geodesics with length $\leq T$ from a point $Q$ to itself. Fix a point $Q$ on $\partial N_p(1)$. We define an injective map $\psi: S_T^n \to G_{T-2\log n+c}(Q)$ for a constant $c$, as follows.
  
   Given $(\gamma,\eta_n)$, substitute $\eta_n$ with an arc going around the cusp once without changing its endpoints on the horocycle $\partial N_p(1)$. Then, move the endpoints of $\eta_n$ on the horocycle to $Q$. See Figure \ref{fig: cusp.exc.count}. We obtain a closed curve $\psi((\gamma,\eta_n))$ (with marked point $Q$) of length $\leq \ell(\gamma)-2\log n+c$, for a constant $c$ only depending on $Q$. 
   The upper bound on the length is derived from Lemma \ref{lem: excursion.length}, coupled with the observation that appending an arc and adjusting its endpoints to $Q$ results in a length change bounded by a constant (dependent only on $Q$).   
   
  \begin{figure}[h]
      \centering

\tikzset{every picture/.style={line width=0.75pt}} 

\begin{tikzpicture}[x=0.65pt,y=0.55pt,yscale=-0.6,xscale=0.6]

\draw    (28,238) .. controls (92,208) and (98,79) .. (92,22) ;
\draw    (98,23) .. controls (93,60) and (115,241) .. (171,228) ;
\draw  [dash pattern={on 4.5pt off 4.5pt}]  (42,230) .. controls (68,248) and (115,257) .. (155,227) ;
\draw  [dash pattern={on 4.5pt off 4.5pt}]  (42,230) .. controls (82,210) and (120,210) .. (155,227) ;
\draw [line width=0.75]    (66,240) .. controls (95,221) and (119,201) .. (120,178) ;
\draw  [dash pattern={on 0.84pt off 2.51pt}]  (76,176) .. controls (87,162) and (114,161) .. (120,178) ;
\draw [line width=0.75]    (76,176) .. controls (97,198) and (114,172) .. (109,134) ;
\draw [line width=0.75]    (67,200) .. controls (70,218) and (102,231) .. (137,236) ;
\draw  [dash pattern={on 0.84pt off 2.51pt}]  (67,200) .. controls (85,177) and (117,193) .. (128,192) ;
\draw [line width=0.75]    (87,138) .. controls (83,167) and (98,221) .. (128,192) ;
\draw  [dash pattern={on 0.84pt off 2.51pt}]  (87,138) .. controls (95,126) and (103,130) .. (109,134) ;
\draw   (186,143) -- (211.44,143) -- (211.44,137) -- (228.4,149) -- (211.44,161) -- (211.44,155) -- (186,155) -- cycle ;
\draw [line width=0.75]    (27,283) .. controls (26,269) and (44,251) .. (66,240) ;
\draw [line width=0.75]    (137,236) .. controls (168,252) and (179,266) .. (186,282) ;
\draw    (245,244) .. controls (309,214) and (315,85) .. (309,28) ;
\draw    (315,29) .. controls (310,66) and (332,247) .. (388,234) ;
\draw  [dash pattern={on 4.5pt off 4.5pt}]  (259,236) .. controls (285,254) and (332,263) .. (372,233) ;
\draw  [dash pattern={on 4.5pt off 4.5pt}]  (259,236) .. controls (299,216) and (337,216) .. (372,233) ;
\draw [line width=0.75]    (283,246) .. controls (312,227) and (342,236) .. (348,205) ;
\draw [line width=0.75]    (284,206) .. controls (287,224) and (319,237) .. (354,242) ;
\draw  [dash pattern={on 0.84pt off 2.51pt}]  (284,206) .. controls (306,187) and (340,195) .. (348,205) ;
\draw [line width=0.75]    (244,289) .. controls (243,275) and (261,257) .. (283,246) ;
\draw [line width=0.75]    (354,242) .. controls (385,258) and (396,272) .. (403,288) ;
\draw  [color={rgb, 255:red, 255; green, 255; blue, 255 }  ,draw opacity=1 ][line width=3] [line join = round][line cap = round] (82,22) .. controls (84.68,22) and (112.66,19.68) .. (107,31) .. controls (106.96,31.09) and (100.26,31.95) .. (100,32) .. controls (93.78,33.24) and (87.33,33.37) .. (81,33) .. controls (80.26,32.96) and (79.75,32) .. (79,32) ;
\draw  [color={rgb, 255:red, 255; green, 255; blue, 255 }  ,draw opacity=1 ][line width=3] [line join = round][line cap = round] (95,31) .. controls (101.67,31) and (108.34,31.33) .. (115,31) .. controls (115.32,30.98) and (108.42,27.47) .. (107,27) .. controls (103.97,25.99) and (91.62,24.34) .. (87,25) .. controls (86.09,25.13) and (90.31,26.83) .. (91,27) .. controls (95.41,28.1) and (98.66,28) .. (104,28) .. controls (106,28) and (99.98,27.72) .. (98,28) .. controls (95.96,28.29) and (90,27.13) .. (90,30) ;
\draw  [color={rgb, 255:red, 255; green, 255; blue, 255 }  ,draw opacity=1 ][line width=3] [line join = round][line cap = round] (346,34) .. controls (335.23,34) and (328.7,36) .. (320,36) ;
\draw  [color={rgb, 255:red, 255; green, 255; blue, 255 }  ,draw opacity=1 ][line width=3] [line join = round][line cap = round] (329,32) .. controls (334.01,32) and (339.02,31.55) .. (344,31) .. controls (346.98,30.67) and (350,31) .. (353,31) .. controls (354.8,31) and (349.79,29.09) .. (348,29) .. controls (339.67,28.56) and (331.34,28) .. (323,28) ;
\draw  [color={rgb, 255:red, 255; green, 255; blue, 255 }  ,draw opacity=1 ][line width=3] [line join = round][line cap = round] (20,30) .. controls (20,89) and (19.67,148) .. (20,207) .. controls (20.03,212.5) and (22.3,220.91) .. (24,226) .. controls (24.06,226.19) and (25,229) .. (25,229) .. controls (25,229) and (25,227.67) .. (25,227) ;
\draw    (438,246) .. controls (502,216) and (508,87) .. (502,30) ;
\draw    (508,31) .. controls (503,68) and (525,249) .. (581,236) ;
\draw  [dash pattern={on 4.5pt off 4.5pt}]  (452,238) .. controls (478,256) and (525,265) .. (565,235) ;
\draw  [dash pattern={on 4.5pt off 4.5pt}]  (452,238) .. controls (492,218) and (530,218) .. (565,235) ;
\draw [line width=0.75]    (510,256) .. controls (536,245) and (540,232) .. (544,214) ;
\draw [line width=0.75]    (472,219) .. controls (476,236) and (486,244) .. (510,256) ;
\draw  [dash pattern={on 0.84pt off 2.51pt}]  (472,219) .. controls (494,197) and (536,204) .. (544,214) ;
\draw [line width=0.75]    (437,291) .. controls (436,277) and (488,267) .. (510,256) ;
\draw [line width=0.75]    (510,256) .. controls (544,265) and (589,274) .. (596,290) ;
\draw  [color={rgb, 255:red, 255; green, 255; blue, 255 }  ,draw opacity=1 ][line width=3] [line join = round][line cap = round] (95,247) .. controls (88.2,247) and (96,247) .. (96,247) .. controls (97.76,248.17) and (94,251.08) .. (94,248) ;
\draw  [color={rgb, 255:red, 0; green, 0; blue, 0 }  ,draw opacity=1 ][line width=3] [line join = round][line cap = round] (98,247) .. controls (98,244.69) and (96,245.72) .. (96,248) ;
\draw  [color={rgb, 255:red, 0; green, 0; blue, 0 }  ,draw opacity=1 ][line width=3] [line join = round][line cap = round] (315,251) .. controls (315,247.76) and (322.86,253) .. (315,253) ;
\draw  [color={rgb, 255:red, 0; green, 0; blue, 0 }  ,draw opacity=1 ][line width=3] [line join = round][line cap = round] (510,256) .. controls (515.04,256) and (510.81,253.09) .. (509,254) .. controls (507.82,254.59) and (506.83,257) .. (510,257) ;
\draw  [color={rgb, 255:red, 255; green, 255; blue, 255 }  ,draw opacity=1 ][line width=3] [line join = round][line cap = round] (307,41) .. controls (314.67,41) and (322.34,41.33) .. (330,41) .. controls (331.2,40.95) and (329,38.67) .. (328,38) .. controls (326.96,37.3) and (324.42,35.35) .. (323,35) .. controls (316.21,33.3) and (310.46,34) .. (302,34) .. controls (300.67,34) and (304.72,33.63) .. (306,34) .. controls (307.87,34.53) and (309.07,36.79) .. (311,37) .. controls (313.98,37.33) and (317.01,36.7) .. (320,37) .. controls (320.47,37.05) and (321.47,37.96) .. (321,38) .. controls (317.61,38.31) and (306.32,38.21) .. (303,36) .. controls (300.91,34.61) and (312.37,29.05) .. (315,28) .. controls (316.28,27.49) and (320.37,27) .. (319,27) .. controls (299.68,27) and (309.62,30) .. (319,30) ;
\draw  [color={rgb, 255:red, 255; green, 255; blue, 255 }  ,draw opacity=1 ][line width=3] [line join = round][line cap = round] (491,43) .. controls (499.33,43) and (507.67,43) .. (516,43) .. controls (517.8,43) and (512.67,41.67) .. (511,41) .. controls (507.14,39.45) and (501.97,37.99) .. (498,37) .. controls (496.67,36.67) and (493.03,36.97) .. (494,36) .. controls (496.53,33.47) and (508.72,32.54) .. (513,32) .. controls (515.01,31.75) and (517.19,31.91) .. (519,31) .. controls (519.3,30.85) and (520.33,31) .. (520,31) .. controls (513.99,31) and (506.25,27.75) .. (502,32) .. controls (501.47,32.53) and (514.29,35.57) .. (516,36) .. controls (517.65,36.41) and (522.7,37) .. (521,37) .. controls (516.17,37) and (511.51,36.29) .. (507,35) .. controls (504.67,34.33) and (502.17,34.09) .. (500,33) .. controls (499.58,32.79) and (498.53,32) .. (499,32) .. controls (503.9,32) and (520.54,38.19) .. (510,39) .. controls (506.01,39.31) and (502,39) .. (498,39) ;
\draw  [color={rgb, 255:red, 255; green, 255; blue, 255 }  ,draw opacity=1 ][line width=3] [line join = round][line cap = round] (97,30) .. controls (93.85,30) and (91.75,32) .. (89,32) ;
\draw   (396,139) -- (421.44,139) -- (421.44,133) -- (438.4,145) -- (421.44,157) -- (421.44,151) -- (396,151) -- cycle ;

\draw (40,157.4) node [anchor=north west][inner sep=0.75pt]    {$E_{n}$};
\draw (88,254.4) node [anchor=north west][inner sep=0.75pt]    {$Q$};
\draw (311,258.4) node [anchor=north west][inner sep=0.75pt]    {$Q$};
\draw (504,261.4) node [anchor=north west][inner sep=0.75pt]    {$Q$};

\end{tikzpicture}
     \caption{How the map $\psi$ modifies a curve by removing an $n-$excursion}
       \label{fig: cusp.exc.count}
  \end{figure}
  
  Note that $\psi$ is injective because if we have a geodesic $\gamma$ with endpoints at $Q$ in the image of $\psi$, then $\psi^{-1}(\gamma)$ is determined uniquely. It is homotopic to the closed curve obtained from joining an $n-$exc to $\gamma$ at $Q$. 
  This implies that $A_n^p(T)=|S_T^n|$ is 
  \begin{equation}
  \leq |G_{T-2\log n+c}|<Ce^{T-2\log n+c}\leq C'\frac{e^T}{n^2},
  \end{equation}
  by Lemma \ref{lemma: mrk.pnt}.  
  \end{proof}
\textbf{Proof of Proposition \ref{prop: n.exc}.} It is implied directly from Proposition \ref{prop: exc.number}. Note that $\ell(\gamma_T) \sim e^T$ by Margulis' result on counting closed geodesics \cite{margulis.aspect}.

\textbf{Proof of Proposition \ref{prop: num.int.cusp}.}
Consider the excursions of $\gamma_T$ with the endpoints on $\partial N_p(1)$. For $r<1$, an $n-$exc enters $N_p(r)$ only when $n>n_r$ for a constant $n_r$ such that $n_r \to \infty$ as $r \to 0$. See Figure \ref{fig: cusp.n_r}.

\begin{figure}[H]
    \centering

\tikzset{every picture/.style={line width=0.75pt}} 

\begin{tikzpicture}[x=0.75pt,y=0.75pt,yscale=-0.7,xscale=0.7]

\draw    (160.2,175.2) -- (466.2,176.2) ;
\draw  [dash pattern={on 4.5pt off 4.5pt}]  (189.8,85.6) -- (390.2,89.2) ;
\draw  [draw opacity=0] (202.55,174.98) .. controls (202.54,174.33) and (202.54,173.68) .. (202.55,173.02) .. controls (203.23,124.74) and (245.07,86.18) .. (295.99,86.9) .. controls (346.65,87.61) and (387.22,126.93) .. (386.97,174.87) -- (294.76,174.32) -- cycle ; \draw   (202.55,174.98) .. controls (202.54,174.33) and (202.54,173.68) .. (202.55,173.02) .. controls (203.23,124.74) and (245.07,86.18) .. (295.99,86.9) .. controls (346.65,87.61) and (387.22,126.93) .. (386.97,174.87) ;  
\draw  [dash pattern={on 0.84pt off 2.51pt}]  (248.2,56.2) -- (249.89,166.87) -- (250,174) ;
\draw  [dash pattern={on 0.84pt off 2.51pt}]  (271.2,58.2) -- (272.81,163.89) -- (273,176) ;
\draw  [dash pattern={on 0.84pt off 2.51pt}]  (319.2,58.2) -- (321,176) ;
\draw [line width=1.5]  [dash pattern={on 1.69pt off 2.76pt}]  (291.1,68.1) -- (312.2,68) ;
\draw  [line width=0.75] [line join = round][line cap = round] (132.13,83.98) .. controls (120.11,84.17) and (123.27,71.53) .. (122.88,67.51) .. controls (122.83,67.03) and (118.82,64.64) .. (118.83,64.64) .. controls (123.31,63.84) and (125.71,63.04) .. (125.76,59.64) .. controls (125.91,48.07) and (125.07,36.52) .. (125.22,24.95) .. controls (125.26,22.04) and (127.82,21) .. (133.16,20.91) ;
\draw  [line width=2.25] [line join = round][line cap = round] (206.37,151.93) .. controls (205.46,151.93) and (202.82,152.93) .. (206.37,152.93) ;
\draw  [line width=2.25] [line join = round][line cap = round] (384.37,155.93) .. controls (384.37,153.57) and (386.72,155.93) .. (384.37,155.93) ;
\draw  [dash pattern={on 4.5pt off 4.5pt}]  (184,153) -- (414.8,156.6) ;
\draw  [line width=0.75] [line join = round][line cap = round] (208.7,50.07) .. controls (208.9,38.05) and (244.41,41.98) .. (255.8,41.84) .. controls (257.14,41.82) and (264.15,37.96) .. (264.15,37.97) .. controls (266.15,42.49) and (268.25,44.94) .. (277.85,45.19) .. controls (310.52,46.05) and (343.22,45.93) .. (375.89,46.79) .. controls (384.1,47) and (386.9,49.63) .. (386.81,54.96) ;
\draw  [dash pattern={on 0.84pt off 2.51pt}]  (226.2,57.2) -- (228,175) ;
\draw  [dash pattern={on 0.84pt off 2.51pt}]  (200.75,57.18) -- (202.55,174.98) ;
\draw  [dash pattern={on 0.84pt off 2.51pt}]  (341.2,58.2) -- (343,176) ;
\draw  [dash pattern={on 0.84pt off 2.51pt}]  (363.2,56.2) -- (365,174) ;
\draw  [dash pattern={on 0.84pt off 2.51pt}]  (385.17,57.07) -- (386.97,174.87) ;

\draw (426,153.4) node [anchor=north west][inner sep=0.75pt]    {$\mathbb{H}$};
\draw (259,15.4) node [anchor=north west][inner sep=0.75pt]    {$n_{r}$};
\draw (330,105.4) node [anchor=north west][inner sep=0.75pt]  [font=\footnotesize]  {$ \begin{array}{l}
E_{n}\\
\end{array}$};
\draw (184,137.4) node [anchor=north west][inner sep=0.75pt]  [font=\scriptsize]  {$z_{1}$};
\draw (386,139.4) node [anchor=north west][inner sep=0.75pt]  [font=\scriptsize]  {$z_{2}$};
\draw (120,142.4) node [anchor=north west][inner sep=0.75pt]  [font=\small]  {$\partial N_{p}( 1)$};
\draw (128,78.4) node [anchor=north west][inner sep=0.75pt]  [font=\small]  {$\partial N_{p}( r)$};
\draw (65,52.4) node [anchor=north west][inner sep=0.75pt]  [font=\small]  {$N_{p}( r)$};

\end{tikzpicture}

    \caption{An $n-$excursion enters $N_p(r)$ whenever $n>n_r$, where $n_r$ is a constant that tends to $\infty$ as $r \to 0$.}
    \label{fig: cusp.n_r}
\end{figure}

We know from Lemma \ref{lem: n,m,exc.int} that the intersection number of an $n-$exc and an $m-$exc is $\leq 2\min(m,n)+2 \leq 4\min(m,n)$ (we can assume $n,m >1$).
Therefore, we have 
\begin{align}
P_*I(\gamma_T,\gamma_T)(N_p(r))\leq \sum \limits_{i,j>n_r} 4\min(i,j) A_i^p(T)A_j^p(T) \nonumber \\
\leq \sum \limits_{i,j>n_r} 4C^2\min(i,j) \frac{e^{2T}}{i^2j^2} \leq 8C^2e^{2T} \sum \limits_{i=n_r}^\infty \frac{1}{i^2}(\sum \limits_{j=n_r}^i \frac{1}{j}) \nonumber \\      
\leq 8C^2e^{2T} \sum \limits_{i=n_r}^\infty \frac{\log i}{i^2}.
\end{align}

The sum $\sum \limits_{i=1}^\infty \log i/i^2$ is finite, so the tail $\sum \limits_{i=n_r}^\infty \log i/i^2$ tends to $0$ when $r \to 0$, as required.

\qed

 \section{Equidistribution}\label{sec: equid.}
 In this section, we prove Theorem \ref{theorem: main}.

Recall that the geodesic current $\gamma_T$ is the sum of the closed geodesics with length $\leq T$, and $L_X$ is the Liouville current of $X$.

We know closed geodesics are equidistributed in $X$ \cite[Thm. 9.1]{margulis.aspect}\cite{Bwn.equi}\cite{roblin}. In other words, we have:

\begin{Thm}(Margulis, Bowen, Parry, Roblin)\label{thm: Margulis.equi}
    Let $X$ be a complete hyperbolic surface of finite area. We have
    \begin{equation}
    \frac{\gamma_T}{\ell(\gamma_T)} \to \frac{L_X}{|L_X|}
    \end{equation}
    as $T \to \infty$, in $C_c^*(T_1(X))$.
\end{Thm}
 
Let $M$ be a metric space. The measures $\{ \mu_a \}$ are called \emph{tight} when for any $\epsilon>0$ there is a compact subset $K_{\epsilon} \subset M$ such that $\mu_a(K_{\epsilon}^c)< \epsilon$ for all indices $a$. The following theorem is a well-known result in measure theory (see, for example, \cite[2.18]{ErgodicTheory}). 
  \begin{Thm} \label{thm: prelim.tight}
Assume that the finite measures $\{\mu_n\}$ on $M$ converge to $\mu$ in $C_c^*(M)$ and are tight. Then the probability measures $\mu_n/|\mu_n|$ converge to $\mu/|\mu|$ in $C_c^*(M)$.     
  \end{Thm}

\textbf{Proof of Theorem \ref{theorem: main}.}
From Theorem \ref{thm: Margulis.equi} and Theorem \ref{thm: cont.I(.,.)}, we conclude as $T \to \infty$

\begin{equation}
\frac{I(\gamma_T,\gamma_T)}{\ell(\gamma_T)^2} \to \frac{I(L_X,L_X)}{|L_X|^2},\label{equ: equi.conv.I_1}
\end{equation}

in $\mathcal{B}^*$. Here, $\mathcal{B}^*$ is the dual of the set of continuous bounded functions with support in a subset $K \cap \mathcal{I}_1(X)$, for some compact subset $K$ of $\mathbb{P}(X) \oplus \mathbb{P}(X)$.  

Therefore, by pushing forward the measures in the convergence sequence (\ref{equ: equi.conv.I_1}) (by map $P: \mathcal{I}_1(X) \to X$), we conclude:
 \begin{equation} 
\frac{1}{{\ell(\gamma_T)}^2}P_*(I(\gamma_T, \gamma_T))  \to \frac{1}{\pi^4 (2g-2+n)^2}P_*(I(L_X, L_X)), 
\end{equation}
 as $T \to \infty$, in $C^*_c(X)$. 
 
 Note that we used the fact that the pullback of a continuous function with compact support on $X$ to $\mathcal{I}_1(X)$ is in $\mathcal{B}_*$.
 
 From Lemma \ref{cor: push Lio} and Lemma \ref{lemma: push c.g}, as $T \to \infty$, we have:
 \begin{equation}
 \frac{1}{{\ell(\gamma_T)}^2}\sum_{p \in \gamma_T \cap \gamma_T} m_p\delta_p \to \frac{\alpha(X)}{2\pi^3(2g-2+n)^2},
 \end{equation}
  in $C_c^*(X)$. The sum is over all the self-intersection points of $\gamma_T$.

These measures are not necessarily probability measures; therefore, to finish the proof, it is enough to show that they are tight (see Theorem \ref{thm: prelim.tight}). Let $K_{\epsilon}$ be the compact subset of $X$ obtained from removing a small horoball neighborhood $N_p(\epsilon)$ of each cusp $p$. The measure of a horoball neighborhood $N_p(\epsilon)$ with respect to $I(\gamma_T,\gamma_T)$ is the number of the intersection points in $N_p(\epsilon)$. On the other hand, we know from Margulis' result that $\ell(\gamma_T) \sim e^T$. Therefore, from Proposition \ref{prop: num.int.cusp}, we conclude that $N_p(\epsilon)$ has measure $<c_{\epsilon}$ with respect to the measure $I(\gamma_T,\gamma_T)/\ell(\gamma_T)^2$. In other words, the complement of $K_{\epsilon}$ has a measure $<c_{\epsilon}$ with respect to these measures and $c_{\epsilon} \to 0$ when $\epsilon \to 0$. This means that these measures are tight, as required.
\qed

\begin{Cor}\label{cor: equi.general}
    The set of the intersection points and two tangent lines at them are equidistributed in $\mathcal{I}_1(X)$ with respect to the measure $I(L_X,L_X)$ which is locally equal to $1/4 \sin(\theta) \, d\theta_1 d\theta_2 \alpha$. Here $\theta$ is the angle between $\theta_1$ and $\theta_2$.
\end{Cor}
\begin{proof}
It is implied from Equation (\ref{equ: equi.conv.I_1}) in the proof of Theorem \ref{theorem: main} and Proposition \ref{prop: meas.I(X)}.
\end{proof}

\section{Geodesic arcs}\label{sec: geod.arc}

Let $\eta$ be a finite-length geodesic arc. Define a measure $\mu_T$ as the sum of the delta measures (considered with multiplicity) at the intersection points between $\eta$ and all the closed geodesics in $\mathcal{G}_T$.

\begin{Cor}\label{cor: segment}
    The intersection points between $\eta$ and the closed geodesics are equidistributed on $\eta$. In other words, the probability measures $\mu_T/|\mu_T|$ converge to the normalized length measure $\ell/\ell(\eta)$ on $\eta$ as $T \to \infty$, in $C_c^*(\eta)$.
\end{Cor}

\textbf{Sketch of the proof.}
The measure corresponding to $\eta$ in $T_1(X)$ is not a geodesic current, but the same proof as Theorem \ref{theorem: main} (the first part of the proof that gives the weaker convergence), works in this case too. Consider the preimage of $\eta$ in $\mathbb{P}(X)$ under the map $P':\mathbb{P}(X)\to X$ and the corresponding transverse measure to $\mathcal{F}$. Note that the assigned measure to a transverse plane $V$ is still invariant when $V$ moves along the leaves of $\mathcal{F}$ in a local neighborhood (as long as it is transverse to $\eta$). Recall that $\mathcal{I}_1(X)$ has two transverse foliations $\mathcal{F}_1$ and $\mathcal{F}_2$. We can see that $P_1^*(\eta)$ induces a measure on each $\mathcal{F}_2$ leaf ($P_1$ is the projection that forgets the second component of $\mathcal{I}_1(X)$). Therefore, we can similarly define the intersection measure between $\eta$ and a geodesic current. Since $\eta$ is compact, an argument similar to Theorem \ref{thm: cont.I(.,.)} implies that the intersection measure between $\eta$ and $\gamma_T$ converges to the intersection measure between $\eta$ and $L_X$  as $T \to \infty$. We know that the push-forward of the intersection measure between $\eta$ and $L_X$ to $\mathbb{P}(X)$ is the length measure of $\eta$ (the length measure of a set $A$ is defined as the length of $A \cap \eta$). For more details, see Proposition \ref{prop: length.LX}. \qed

\bibliography{biblio}

\begin{thebibliography}{KKH}

\bibitem[Bas]{Basmj}
Ara Basmajian.
\newblock {Universal length bounds for non-simple closed geodesics on
  hyperbolic surfaces}.
\newblock {\em Journal of Topology} {\bf 6}(2013), 513--524.

\bibitem[Bon1]{Bon.gc.3}
Francis Bonahon.
\newblock {Bouts des vari{\'e}t{\'e}s hyperboliques de dimension 3}.
\newblock {\em Annals of Mathematics} {\bf 124}(1986), 71--158.

\bibitem[Bon2]{Bon.gc}
Francis Bonahon.
\newblock {The geometry of Teichm{\"u}ller space via geodesic currents}.
\newblock {\em Inventiones mathematicae} {\bf 92}(1988), 139--162.

\bibitem[Bow]{Bwn.equi}
Rufus Bowen.
\newblock {The equidistribution of closed geodesics}.
\newblock {\em American Journal of Mathematics} {\bf 94}(1972), 413--423.

\bibitem[Bus]{Bsr}
Peter Buser.
\newblock {\em Geometry and spectra of compact Riemann surfaces}.
\newblock Springer Science \& Business Media, 2010.

\bibitem[KKH]{katokBook}
Anatole Katok, AB~Katok, and Boris Hasselblatt.
\newblock {\em Introduction to the modern theory of dynamical systems}.
\newblock Number~54. Cambridge university press, 1995.

\bibitem[Lal]{lalleyEqui}
Steven~P Lalley.
\newblock {Self-intersections of closed geodesics on a negatively curved
  surface: statistical regularities}.
\newblock {\em Convergence in ergodic theory and probability (Columbus, OH,
  1993)} {\bf 5}(1996), 263--272.

\bibitem[Mar]{margulis.aspect}
Gregori~Aleksandrovitsch Margulis.
\newblock {On some aspects of the theory of Anosov systems}.
\newblock In {\em On Some Aspects of the Theory of Anosov Systems}, pages
  1--71. Springer, 2004.

\bibitem[Rob]{roblin}
Thomas Roblin.
\newblock {Ergodicit{\'e} et {\'e}quidistribution en courbure n{\'e}gative}.
\newblock (2003).

\bibitem[Tor]{Tina.sys}
Tina Torkaman.
\newblock {Intersection Number, Length, and Systole on Compact Hyperbolic
  Surfaces}.
\newblock {\em arXiv preprint arXiv:2306.09249} (2023).

\bibitem[VO]{ErgodicTheory}
Marcelo Viana and Krerley Oliveira.
\newblock {\em Foundations of ergodic theory}.
\newblock Number 151. Cambridge University Press, 2016.

\end{thebibliography}
\bibliographystyle{math}

\end{document}